\documentclass[12pt]{amsart}

\usepackage[letterpaper,left=1in,right=1in,top=0.8in,bottom=0.8in]{geometry}
\usepackage{amsmath,amssymb,amsthm,color}
\usepackage{hyperref}
\usepackage{mathtools}

\newtheorem{theorem}{Theorem}[section]
\newtheorem{proposition}[theorem]{Proposition}
\newtheorem{lemma}[theorem]{Lemma}

\newtheorem*{MainProblem}{Lattice Triangle Problem (Veech \cite{Veech1})}
\newtheorem*{conjecture}{Conjecture}

\theoremstyle{definition}

\newtheorem{remark}{Remark}
\newtheorem*{example}{Example}

\newcommand{\Z}{\mathbb{Z}}
\newcommand{\e}{\mathrm{e}}
\newcommand{\R}{\mathbb{R}}
\newcommand{\N}{\mathbb{N}}
\newcommand{\SL}{\mathrm{SL}}

\hypersetup{hidelinks}

\title{On the paucity of lattice triangles}

\subjclass[2020]{37E35, 30F60}
\keywords{Veech surfaces, lattice triangles, Teichm\"uller curves, lattice points}

\author[D.~K. Angdinata]{David Kurniadi Angdinata}
\author[E. Chen]{Evan Chen}
\author[K. Ono]{Ken Ono}
\author[J. Thorner]{Jesse Thorner}
\author[Jiaxin Zhang]{Jiaxin Zhang}
\author[Jujian Zhang]{Jujian Zhang}

  \thanks{\textsuperscript{$\times$}Corresponding author. Email: \texttt{evan@axiommath.ai}}
  \email{evan@axiommath.ai}
  \email{ken@axiommath.ai}
  \email{jesse@axiommath.ai}
  \email{jat9@illinois.edu}
  \email{jiaxin@axiommath.ai}
  \email{jujian@axiommath.ai}

\begin{document}

\begin{abstract}
A rational triangle $T$ (one whose angles are rational multiples of $\pi$) unfolds to a translation surface ${X_T}$.  The \emph{lattice triangle problem} asks to classify those $T$ for which ${X_T}$ is a Veech (lattice) surface, which means that the $\SL_2(\R)$-orbit of ${X_T}$ is closed in its stratum (so its projection to moduli space is a Teichm\"uller curve). The most mysterious regime is the ``hard obtuse window'' (largest angle in $(\pi/2,2\pi/3]$), where it is conjectured that no lattice triangles exist. Using an arithmetic reformulation of the Mirzakhani-Wright rank obstruction, we prove a quantitative theorem that rules out all but a proportion $n^{-1+o(1)}$ of the triangles in this window with denominator $n$.  The main technical result in our proof was autoformalized by AxiomProver in Lean (using mathlib).
\end{abstract}

\maketitle

\section{Introduction and statement of results}\label{sec:Intro1}

The classification of lattice triangles is a central problem in the study of translation surfaces and the dynamics of rational billiards \cite{GJ, Hub, Vo}. A rational triangle $T$ (one whose interior angles are rational multiples of $\pi$) unfolds to a translation surface ${X_T}$ \cite{GJ, Wa}.  The \emph{lattice triangle problem} asks for those $T$ whose unfolding is a Veech (or lattice) surface \cite{Veech1, Veech2}, meaning that its affine automorphism group is a lattice in $\SL_2(\R)$ \cite{Hub, Veech1, Veech2}.  The $\SL_2(\R)$-orbits of such surfaces project to Teichm\"uller curves in the moduli space of abelian differentials \cite{E, EMM, McM, MW}\footnote{The Eskin-Mirzakhani-Mohammadi paper \cite{EMM} includes the celebrated ``Magic Wand Theorem" that gives the classification of orbit closures for the $\SL_2(\R)$-action on the moduli spaces of translation surfaces.}.  The acute and right-angled cases are fully classified \cite{KS00, SP11}.  This paper studies the obtuse scalene regime, which remains a challenge \cite{LNZ, MW}.

Let $\N$ denote the positive integers.  A triangle $T$ is \emph{rational} if its interior angles
\[
\Big(\frac{p\pi}{n}, \frac{q\pi}{n}, \frac{r\pi}{n}\Big),
\]
satisfy $p, q, r \in {\N}$, $p+q+r=n$, and $\gcd(p, q, r, n)=1$.  The billiard flow\footnote{Recall that the \emph{billiard flow} on a Euclidean triangle $T$ is the motion of a point mass traveling at constant speed along straight paths, reflecting off the boundary so that the angle of incidence equals the angle of reflection.  We note that in the setting billiard flows, the holomorphic 1-form $\omega_T$ is uniquely determined by the unfolding process, so we suppress $\omega_T$.} on a rational triangle $T$ can be ``unfolded" into a linear flow on a translation surface ${X_T}$ as follows.
\begin{enumerate}
    \item Start with $T$ in the plane and reflect it repeatedly in its sides.
    \item The angles are rational multiples of $\pi$, so only finitely many orientations of $T$ arise.
    \item Gluing the parallel sides of the resulting copies by translations produces a compact translation surface $X_T$.
\end{enumerate}
The triangle $T$ is called a \emph{lattice triangle} if ${X_T}$ is a \emph{Veech surface}, that is, if its Veech group (the image of the derivative map from the affine automorphism group) is a lattice in $\SL_2(\R)$.  Equivalently, the unfolding generates a Teichm\"uller curve in the moduli space of translation surfaces.

\begin{example}
Consider the isosceles triangle $T$ with angles $(\pi/5, \pi/5, 3\pi/5)$, so that $n=5$ and $(p,q,r) = (1,1,3)$.  Unfolding generates $2n = 10$ copies of $T$ which, when glued along translated sides, form a translation surface ${X_T}$ of genus $g=2$ whose unique zero has order $2$ \cite{McM, Veech1}.  It is a double cover of a regular pentagon branched at the center.  The affine automorphism group of $X_T$ is the Hecke triangle group $G_5$.  This is a lattice in $\SL_2(\R)$ \cite{Veech1, Veech2}, so $T$ is a lattice triangle.
\end{example}

\begin{MainProblem}
Classify the lattice rational triangles.
\end{MainProblem}

This problem has been fully resolved in the acute and right-angled cases:  Any such triangle that is neither sporadic nor a member of a known family is not a lattice triangle (see \cite{Hub} and the references therein).  In the obtuse case, it is known \cite{LNZ} that there are two infinite families of rational lattice triangles, namely those with angles
\[
\Big(\frac{\pi}{n},\frac{\pi}{n},\frac{(n-2)\pi}{n}\Big) \qquad\text{and}\qquad
\Big(\frac{\pi}{2n},\frac{\pi}{n},\frac{(2n-3)\pi}{2n}\Big),
\]
together with Hooper's sporadic obtuse scalene example \cite{Hooper} with angles $(\frac{\pi}{12},\frac{\pi}{3},\frac{7\pi}{12})$.

\begin{remark}
The small number of sporadic lattice triangles generally occur in low-genus strata of the moduli space of abelian differentials \cite{McM, MW}.
\end{remark}

These examples motivate the following conjecture (see \cite{LNZ}).

\begin{conjecture}
The known list of obtuse rational lattice triangles is complete. Equivalently, every obtuse rational lattice triangle is either a member of one of the two known infinite families or is Hooper's triangle.
\end{conjecture}

The point of the conjecture is that the primary open challenge is the \emph{obtuse scalene}\footnote{Note that the example obtuse triangle with angles $(\pi/5, \pi/5, 3\pi/5)$ is isosceles.} regime (i.e.\ distinct interior angles), and that the hardest case is that of triangles whose obtuse angle lies in the ``hard window'' $(\pi/2, 2\pi/3]$ \cite{LNZ, MW}.  The primary tool for ruling out lattice triangles is the affine invariant manifold theory of Mirzakhani and Wright \cite{MW}, specifically their \emph{rank obstruction}.  Larsen, Norton, and Zykoski \cite{LNZ} reformulated it into an effective number-theoretic condition (see Proposition~\ref{prop:LNZ} below) and thereby ruled out lattice triangles in the ``strongly obtuse'' regime, where the obtuse angle exceeds $2\pi/3$.  It is natural to ask how well this criterion handles the ``hard window.'' Indeed, this paper is inspired by an e-mail from Alex Wright 
\cite{WrightEmail}:

\medskip
\noindent
{\it ``Mathematicians are sure they know all lattice triangles; the problem is to *prove* all other triangles are not lattice triangles. There are some elementary criteria. The one used in the paper \cite{LNZ} is just in terms of modular arithmetic. That criterion isn't enough on its own to completely finish the problem, but it should be possible to get surprisingly close using just that criteria.''}
\medskip

Here, we quantitatively confirm his speculation, demonstrating that this number theoretic criterion is powerful enough to disqualify all but a proportion $n^{-1+o(1)}$ of the triangles with denominator $n$ in this regime.  To describe our result, we fix some notation.  For $n\in\mathbb{N}$, define the sets
\[
\begin{aligned}
\mathcal{T}_n&\coloneq \{(p,q)\in\mathbb{N}\times\mathbb{N}\colon p+q<\tfrac{n}{2}\},\\
\mathcal H_n &\coloneq \{(p,q)\in \mathcal{T}_n\colon \gcd(p,q,n)=1\},\\
\mathcal{L}_n&\coloneq \{(p,q)\in\mathcal{H}_n\colon \textup{$(p,q)$ corresponds with the lattice triangle $(\tfrac{p\pi}{n},\tfrac{q\pi}{n},\tfrac{(n-p-q)\pi}{n})$}\}.
\end{aligned}
\]
Since $p+q<\tfrac{n}{2}$ is equivalent to $r>\frac{n}{2}$, the region $\mathcal{H}_n$ is the obtuse region.  We assume that $n\geq 5$; otherwise, the hypothesis that $p,q\geq 1$ implies that $\mathcal{T}_n$ is empty.

Our main result shows that the number-theoretic implementation of the Mirzakhani--Wright rank obstruction is remarkably powerful, ruling out the existence of lattice triangles for almost all candidates in the elusive hard window.  In particular, we prove that if $n$ is large, then $\mathcal{L}_n$ is a thin subset of $\mathcal{H}_n$

\begin{theorem}
\label{thm:Sharp}
For all $\varepsilon>0$, there exists a constant $c(\varepsilon)>0$ such that if $n\geq 5$ is an integer, then $\#\mathcal{L}_n / \#\mathcal{H}_n\leq c(\varepsilon)/n^{1-\varepsilon}$.
\end{theorem}

Our proof begins with the work of Larsen, Norton, and Zykoski (see Proposition~\ref{prop:LNZ}), which gives a sufficient criterion for $(p,q)\in\mathcal{H}_n$ to not correspond with a lattice triangle.  Their condition is purely number-theoretic, and our proof of Theorem~{thm:Sharp} is purely an analysis of their condition.  Without additional input, the exponent $-1$ is optimal; see Remark~\ref{rem:optimal}.

\subsection*{Acknowledgements}

We thank Anne Larsen and Alex Wright for comments on an earlier version of this manuscript and the referee for several illuminating questions and comments.
This paper describes a test case for AxiomProver, an autonomous system.  The project engineering team is Chris Cummins, Ben Eltschig, GSM, Dejan Grubisic, Leopold Haller, Letong Hong (principal investigator), Andranik Kurghinyan, Kenny Lau, Hugh Leather, Simon Mahns, Aram H. Markosyan, Rithikesh Muddana, Manooshree Patel, Gaurang Pendharkar, Vedant Rathi, Alex Schneidman, Volker Seeker, Shubho Sengupta (principal investigator), Ishan Sinha, Jimmy Xin, and Jujian Zhang.

\section{Number-theoretic form of the Mirzakhani--Wright rank obstruction}
\label{sec:RankObstruct}

In this section, we transition from the geometric definition of lattice triangles to the number-theoretic framework required to prove Theorem~\ref{thm:Sharp}.  The classification of lattice surfaces interfaces with the study of orbit closures in the moduli space of abelian differentials \cite{E, EMM, MW}.  Mirzakhani and Wright supplied a definitive geometric tool for this study through their theory of \emph{rank}, which constrains the affine invariant submanifolds that a lattice surface can inhabit.  For rational triangles, Larsen, Norton, and Zykoski \cite{LNZ} distilled this abstract obstruction into a concrete criterion that reduces the lattice property to modular arithmetic.

To state this reformulation, we {incorporate} some notation.  For a positive integer $d$, write $\Z/d\Z$ for the group of residues modulo $d$, and, for $x\in\Z$ or $x\in\Z/d\Z$, write $[x]_d\in\{0,1,\ldots,d-1\}$ for the least nonnegative residue of $x$ modulo $d$.  Let
\[
U_n \coloneq \{a\in\Z/n\Z\colon\gcd(a,n)=1\},\qquad \varphi(n)\coloneq \# U_n.
\]
A unit $a\in U_n$ is called \emph{usable} if $2a\not\equiv 2\pmod n$.  Larsen, Norton, and Zykoski \cite{LNZ} obtained the following crucial criterion.

\begin{proposition}[{\cite[Proposition 2.1]{LNZ}}]
\label{prop:LNZ}
Let $T$ be a rational obtuse triangle with angles $(\frac{p\pi}{n},\frac{q\pi}{n},\frac{r\pi}{n})$ in lowest terms, and let ${X_T}$ be its translation surface.  If there exists a usable unit $a\in U_n$ such that at least two of the {inequalities}
\[
[ap]_n<[2p]_n,\qquad [aq]_n<[2q]_n,\qquad [ar]_n<[2r]_n
\]
{hold, t}hen ${X_T}$ does not have the lattice property. Consequently, $(p,q)\in\mathcal{H}_n\smallsetminus\mathcal{L}_n$.
\end{proposition}

We will show that for almost all pairs $(p,q)\in\mathcal{H}_n$ one may find a usable $a$ such that
\[
[ap]_n<[2p]_n\qquad \text{and}\qquad [aq]_n<[2q]_n,
\]
in which case Proposition~\ref{prop:LNZ} can be applied. For an integer $0\le h\le n-1$, define
\[
I_{h} \coloneq \{1,2,\dots,h\},\qquad   \mathbf{1}_{I_{h}}(x) \coloneq
  \begin{cases}
    1,& 1\le [x]_n\le h,\\
    0,& \text{otherwise.}
  \end{cases}
\]
For $(p,q)\in\mathcal{H}_n$, we define
\[
h_p \coloneq 2p-1\qquad \text{and} \qquad h_q \coloneq 2q-1.
\]
We have $h_p,h_q < n$ since we are assuming our triangle is obtuse. Define
\[
S(p,q) \coloneq \#\{a\in U_n\colon \textup{$[ap]_n<[2p]_n$ and $[aq]_n<[2q]_n$}\}.
\]

\begin{lemma}
\label{lem:Sidentity}
If $n\ge 5$, $(p,q)\in\mathcal{T}_n$, and $a\in U_n$, then
\[
\mathbf{1}_{I_{h_p}}(ap)=1\iff [ap]_n<[2p]_n, \qquad \mathbf{1}_{I_{h_q}}(aq)=1\iff [aq]_n<[2q]_n .
\]
In particular, we have
\[
S(p,q) = \sum_{a\in U_n} \mathbf{1}_{I_{h_p}}(ap) \mathbf{1}_{I_{h_q}}(aq).
\]
\end{lemma}

\begin{proof}
Since $q\ge 1$ and $p+q<\frac n2$, we have $p<\frac n2-1$, hence $1\le 2p<n-2$ and therefore
$[2p]_n=2p$.  Moreover $\gcd(a,n)=1$ and $1\le p<n$, so $n\nmid ap$ and hence $[ap]_n\ge 1$.
Consequently $\mathbf{1}_{I_{h_p}}(ap)=1$ if and only if $1\le [ap]_n\le 2p-1$, that is, if and
only if $[ap]_n<2p=[2p]_n$.  The assertion for $q$ is proved in the same way.
\end{proof}

The notion of ``usable'' in Proposition~\ref{prop:LNZ} provides an approach to bound $\#\mathcal{L}_n$.

\begin{lemma}
\label{lem:usable}
If $n\ge 3$, then $\mathcal{L}_n\subseteq\{(p,q)\in\mathcal{H}_n\colon S(p,q)\le 2\}$.
\end{lemma}

\begin{proof}
A unit $a\in U_n$ fails to be usable if and only if $2a\equiv 2\pmod{n}$, which is equivalent to
$a\equiv 1\pmod{n/\gcd(2,n)}$.  If $n$ is odd, the only such $a\in U_n$ is $a=1$; if $n$ is even,
then $a\in\{1,1+\frac n2\}$.  In either case, there are at most two such units.  If 
$S(p,q)\ge 3$, then by Lemma~\ref{lem:Sidentity}, the set of $a\in U_n$ counted by $S(p,q)$ has at
least three elements, and thus contains a usable unit $a$.  This $a$ satisfies two of the three
inequalities of Proposition~\ref{prop:LNZ}, hence 
$(p,q)\in\mathcal{H}_n\smallsetminus\mathcal{L}_n$.
\end{proof}


\section{Notation and preliminary lemmas}
\label{sec:Prelim}

We will use several conventions and elementary facts from multiplicative number theory \cite{HW}.  We write $\tau(n)$ for the number of divisors
of $n$, $\omega(n)$ for the number of distinct primes dividing $n$, and
\[
\mathcal{Q} = \mathcal{Q}(n)\coloneq 2^{\omega(n)}(1+\log n)^4 .
\]
We have the classical bounds
\[
\omega(n)\leq 2\frac{\log n}{\log\log n},\qquad \tau(n)\leq \exp\Big(2\frac{\log n}{\log\log n}\Big).
\]
It follows that if $\varepsilon>0$, then
\begin{equation}
\label{eqn:Q_small}
2\le\mathcal{Q}\ll_{\varepsilon}n^{\varepsilon}.
\end{equation}
We also use the following elementary bound for all $n\in\mathbb{N}$:
\begin{equation}
\label{eqn:phi_lower}
\frac{n}{\varphi(n)}=\prod_{\substack{r\mid n \\ \textup{$r$ prime}}}\frac{p}{p-1}\leq 2^{\omega(n)}.
\end{equation}
Finally, we set
\[
M(p,q)\coloneq \frac{h_ph_q}{n^2}\varphi(n)=\frac{(2p-1)(2q-1)}{n^2}\varphi(n),
\qquad E(p,q)\coloneq S(p,q)-M(p,q),
\]
so that $M(p,q)$ is the value that $S(p,q)$ would take if the pairs $([ap]_n,[aq]_n)$ with
$a\in U_n$ were equidistributed in $(\Z/n\Z)\times(\Z/n\Z)$.

Define the weights
\[
w_N(k)\coloneq \frac{1}{2\min(k,N-k)}\qquad (N\ge 2,\ 1\le k\le N-1).
\]
We will frequently use the following elementary estimate.

\begin{lemma}
\label{lem:harmonic}
If $n\ge 2$, then $\sum_{k=1}^{n-1}w_n(k)\le \log n$.
\end{lemma}

\begin{proof}
Write $H_M\coloneq\sum_{j\le M}\frac1j$, with the convention $H_0=0$.  We use the classical estimate
\begin{equation}
\label{eqn:harmonic_number}
H_M\le \log M+\min\Big\{\gamma+\frac{1}{2M},1\Big\}\qquad(M\ge 1).
\end{equation}
If $n$ is odd and $M$ satisfies $n=2M+1$, then as $k$ runs through $1,2,\dots,n-1$, the quantity $\min(k,n-k)$ assumes each of the values $1,2,\dots,M$ exactly twice.  It follows that $\sum_{k=1}^{n-1}w_n(k)=H_M$, so the lemma reduces to the bound $H_M\le\log(2M+1)$.  This is straightforward to check.

If $n$ is even and $M$ satisfies $n=2M$, then $\min(k,n-k)$ assumes each of the values $1,\dots,M-1$ exactly twice and the value $M$ once.  It follows that $\sum_{k=1}^{n-1}w_n(k)=H_{M-1}+\frac{1}{2M}$, and the lemma reduces to the bound $H_{M-1}+\frac{1}{2M}\le\log(2M)$.  This is also straightforward to check.
\end{proof}

We need estimates for the Fourier coefficients of the indicator function of an interval.

\begin{lemma}
\label{lem:intervalFourier}
Let $d\ge2$, let $0\le H\le d-1$, let $f\coloneq \mathbf{1}_{\{1,2,\dots,H\}}$ as a function on $\Z/d\Z$, and let $\widehat{f}(k)\coloneq \frac1d\sum_{x\in\Z/d\Z}f(x)\e^{-2\pi ikx/d}$.  Then $\widehat{f}(0)=H/d$, and
\[
|\widehat{f}(k)|\le\min\Big(\frac Hd,\ w_d(k)\Big)\qquad(1\le k\le d-1).
\]
\end{lemma}

\begin{proof}
The formula for $\widehat{f}(0)$ is immediate, as is the bound $|\widehat{f}(k)|\le H/d$.  Let $1\le k\le d-1$.  Summing the geometric series, we obtain
\[
|\widehat{f}(k)|=\frac1d\Big|\sum_{x=1}^{H}\e^{-2\pi ikx/d}\Big| =\frac{1}{d}\Big|\frac{\sin(H k \pi/d)}{\sin(k\pi/d)}\Big|\leq \frac{1}{d|\sin(k \pi/d)|}.
\]
Note that $|\sin(\pi x)|\ge2\min(x,1-x)$ for $x\in[0,1]$.  Taking $x=k/d$, we find that $d|\sin(\pi k/d)|\ge2\min(k,d-k)$, hence $|\widehat{f}(k)|\le w_d(k)$.
\end{proof}

Finally, we record the order of magnitude of $\#\mathcal{H}_n$.

\begin{lemma}
\label{lem:sizeHn}
If $n\geq 5$, then $\#\mathcal{H}_n\asymp n^2$.
\end{lemma}

\begin{proof}
For $e\in\N$, the map $(p,q)\mapsto(p/e,q/e)$ is a bijection from $\{(p,q)\in\mathcal{T}_n\colon e\mid p,~e\mid q\}$ onto $\{(p',q')\in\N^2\colon p'+q'<\frac{n}{2e}\}$.  The latter set has cardinality $\binom{K_e}{2}$, where $K_e= \lceil\frac{n}{2e}\rceil-1$.  It follows that $\binom{K_e}{2}=\frac{n^2}{8e^2}+O(\frac ne)$.  Detecting the condition $\gcd(p,q,n)=1$ by M\"obius inversion over the divisors $e$ of $n$, we conclude that
\[
\#\mathcal{H}_n=\sum_{e\mid n}\mu(e)\binom{K_e}{2} =\frac{n^2}{8}\sum_{e\mid n}\frac{\mu(e)}{e^2}+O\Big(n\sum_{e\mid n}\frac1e\Big)=\frac{n^2}{8}\prod_{r\mid n}(1-r^{-2})+O(n\log n),
\]
where $r$ runs over the primes dividing $n$.  The product over $r\mid n$ lies between $\zeta(2)^{-1}$ and $1$, so the main term is $\asymp n^2$ and dominates the error term for $n$ large enough.  For the remaining $n$ it suffices to note that $(1,1)\in\mathcal{H}_n$, so that $\#\mathcal{H}_n\ge1$.
\end{proof}

\section{Proof of Theorem~\texorpdfstring{\ref{thm:Sharp}}{1.1}}
\label{sec:Sharper}

The quantity $E(p,q)$ is a discrepancy, measuring the failure of the point set
\[
\{([ap]_n,[aq]_n)\colon a\in U_n\}
\]
to intersect $[1,h_p]\times[1,h_q]$ at the expected number of points.  A direct approach to estimating $E(p,q)$ is difficult because the units $a\in U_n$ have a complicated distribution within $\Z/n\Z$.  To describe our approach, we let
\[
\mathcal{D}_n\coloneq \{d\in\N\colon n/d\ \textup{is squarefree, $d\mid n$}\},
\]
so that $\#\mathcal{D}_n=2^{\omega(n)}$.  We decompose $S(p,q)$ exactly over the divisors of $n$, trading the coprimality condition on $a$ for a family of $2^{\omega(n)}$ counting problems, one for each $d\in\mathcal{D}_n$, whose point sets $\{([bp]_d,[bq]_d)\colon b\in\Z/d\Z\}$ are cyclic subgroups of $(\Z/d\Z)\times(\Z/d\Z)$.  The discrepancy of each is governed by the short vectors of the dual lattice $\Lambda_d$ (see \eqref{eqn:Lambda_d_def}).  Such vectors are rare.

Throughout, $n\ge 5$ and $(p,q)\in\mathcal{H}_n$.  Recall that $h_p=2p-1$ and
$h_q=2q-1$, so that
\begin{equation}
\label{eqn:hsum}
1\le p\le h_p,\qquad 1\le q\le h_q,\qquad h_p+h_q<n.
\end{equation}
  For $d\in\mathcal{D}_n$, let
\[
H_p^{(d)}\coloneq \Big\lfloor\frac{h_p d}{n}\Big\rfloor,\quad H_q^{(d)}\coloneq \Big\lfloor\frac{h_q d}{n}\Big\rfloor,\quad
N_d\coloneq \#\big\{b\in\Z/d\Z\colon 1\le[bp]_d\le H_p^{(d)},\ 1\le[bq]_d\le H_q^{(d)}\big\},
\]
Let $\mathcal{F}_d\coloneq \{j\in\Z\colon -d/2<j\le d/2\}$, a
complete set of representatives for $\Z/d\Z$.  If $j\in\mathcal{F}_d$, then $|j|$ equals the
distance from $j$ to $0$ modulo $d$.

\subsection{The divisor decomposition}

\begin{lemma}
\label{lem:divdecomp}
If $n\ge5$ and $(p,q)\in\mathcal{T}_n$, then $S(p,q)=\sum_{d\in\mathcal{D}_n}\mu(n/d)N_d$.
\end{lemma}

\begin{proof}
Since $\sum_{c\mid m}\mu(c)=\mathbf{1}_{m=1}$, Lemma~\ref{lem:Sidentity} gives
\[
S(p,q)=\sum_{c\mid n}\mu(c) \#\{a\in[1,n]\colon c\mid a,\ 1\le[ap]_n\le h_p,\ 1\le[aq]_n\le h_q\}.
\]
Write $a=cb$ with $1\le b\le d$, where $d\coloneq n/c$.  Since $n=cd$, we have $[cbp]_n=c [bp]_d$, so
the condition $1\le[cbp]_n\le h_p$ holds if and only if $1\le[bp]_d\le\lfloor h_p/c\rfloor=H_p^{(d)}$.
The inner cardinality therefore equals $N_{n/c}$, and the terms for which $n/d$ is not squarefree
vanish.
\end{proof}

\subsection{A majorant for each divisor}

For $d\in\mathcal{D}_n$, let
\begin{equation}
\label{eqn:Lambda_d_def}
\Lambda_d\coloneq \{(k,\ell)\in\Z^2\colon kp+\ell q\equiv0\pmod d\}.
\end{equation}
Since $\gcd(p,q,d)$ divides $\gcd(p,q,n)=1$, the map $(k,\ell)\mapsto kp+\ell q$ from
$(\Z/d\Z)\times(\Z/d\Z)$ to $\Z/d\Z$ is surjective, hence
\begin{equation}
\label{eqn:detLambda}
\det\Lambda_d=[\Z^2:\Lambda_d]=d.
\end{equation}
Define $u_d$ and $v_d$ on $\mathcal{F}_d$ by $u_d(0)\coloneq H_p^{(d)}$, $v_d(0)\coloneq H_q^{(d)}$.  Also, for nonzero integers $k$ and $\ell$, define
\[
u_d(k)\coloneq \frac{d}{2|k|},\qquad v_d(\ell)\coloneq \frac{d}{2|\ell|},\qquad  \mathcal{W}_d\coloneq \frac1d\sum_{\substack{(k,\ell)\in\Lambda_d\cap(\mathcal{F}_d\times \mathcal{F}_d)\\ (k,\ell)\ne(0,0)}}u_d(k) v_d(\ell).
\]

\begin{lemma}
\label{lem:majd}
If $n\ge5$, $(p,q)\in\mathcal{H}_n$, and $d\in\mathcal{D}_n$, then
\[
\Big|N_d-\frac{H^{(d)}_pH^{(d)}_q}{d}\Big|\le\mathcal{W}_d, \qquad\text{and}\qquad |E(p,q)|\le 2^{\omega(n)+1}+\sum_{d\in\mathcal{D}_n}\mathcal{W}_d.
\]
\end{lemma}

\begin{proof}
If $d=1$, then \eqref{eqn:hsum} implies that $H^{(d)}_p=\lfloor h_p/n\rfloor=0$, hence $N_d=0$.  As $\mathcal{F}_1^2\smallsetminus\{(0,0)\}$ is empty, $\mathcal{W}_1=0$ as well.  The first assertion therefore reads $0\le0$.

Assume then that $d\ge2$.  Let
\[
f_p\coloneq \mathbf{1}_{\{1,\dots,H^{(d)}_p\}},\qquad f_q\coloneq \mathbf{1}_{\{1,\dots,H^{(d)}_q\}}
\]
as functions on $\Z/d\Z$, and let
\[
\widehat{f}(k)=\frac1d\sum_{x\in \Z/d\Z}f(x)\e^{-2\pi ikx/d}.
\]
Fourier inversion and the orthogonality relation
\[
\sum_{b\in\Z/d\Z}\e^{2\pi ib(kp+\ell q)/d}=d \mathbf{1}_{(k,\ell)\in\Lambda_d}
\]
combine to give
\[
N_d=\sum_{b\in\Z/d\Z}f_p(bp)f_q(bq)=d\!\!\sum_{(k,\ell)\in\Lambda_d\cap(\mathcal{F}_d\times \mathcal{F}_d)}\!\!\widehat{f_p}(k) \widehat{f_q}(\ell).
\]
The $(k,\ell)=(0,0)$ term equals $H^{(d)}_pH^{(d)}_q/d$.  By Lemma~\ref{lem:intervalFourier}, we have $|\widehat{f_p}(k)|\le u_d(k)/d$ for all $k\in\mathcal{F}_d$, with equality at $k=0$, and likewise for $g_q$.  Now, the first assertion follows from the triangle inequality.

For the second assertion, the identity $\sum_{d\mid n}\mu(n/d)d=\varphi(n)$ implies that
\[
M(p,q)=\frac{h_ph_q}{n^2}\sum_{d\mid n}\mu(n/d).
\]
Therefore, Lemma~\ref{lem:divdecomp} gives
\[
|E(p,q)|\le\sum_{d\in\mathcal{D}_n}\Big|N_d-\frac{h_ph_qd}{n^2}\Big| \le\sum_{d\in\mathcal{D}_n}\Big(\mathcal{W}_d+\Big|\frac{H^{(d)}_pH^{(d)}_q}{d}-\frac{h_ph_qd}{n^2}\Big|\Big).
\]
If $\theta_p,\theta_q\in[0,1)$ satisfy $H^{(d)}_p=\frac{h_pd}{n}-\theta_p$ and $H^{(d)}_q=\frac{h_qd}{n}-\theta_q$, then the second summand equals
\[
\Big|\frac{\theta_qh_p}{n}+\frac{\theta_ph_q}{n}-\frac{\theta_p\theta_q}{d}\Big|\leq \frac{h_p+h_q}{n}+\frac1d\le2
\]
by \eqref{eqn:hsum}.  Since $\#\mathcal{D}_n=2^{\omega(n)}$, the
proof is complete.
\end{proof}

\subsection{Lattices and resonance}

We require a result from lattice theory.  The proof is elementary, but we could not find a reference.

\begin{lemma}
\label{lem:lattice}
Let $A,B>0$ and $R\coloneq \{(x,y)\colon|x|\le A,\ |y|\le B\}$.  Let $\Lambda\subseteq\R^2$ be a lattice of determinant $D$ and  $J\coloneq \#\big((\Lambda\cap R)\smallsetminus\{0\}\big)$.  If $J\ge1$, then either $J<24AB/D$, or there exists $w=(w_1,w_2)\in\Lambda\smallsetminus\{0\}$ such that $|w_1|\le4A/J$ and $|w_2|\le4B/J$.
\end{lemma}

\begin{proof}
Define $\|(x,y)\|\coloneq \max(|x|/A,|y|/B)$ for $(x,y)\in\R^2$, so that $R$ is the unit ball of $\|\cdot\|$.  Let
\[
\lambda\coloneq \min\{\|z\|\colon z\in\Lambda\smallsetminus\{0\}\},
\]
which is attained at some $w^{\ast}=(s,t)$.  Since $J\ge1$, we have $\lambda\le1$.  It follows that $w^{\ast}$ is primitive (since $w^{\ast}=mz$ with $z\in\Lambda$ and $m\ge2$ would give $\|z\|<\lambda$) and $\Lambda\cap\R w^{\ast}=\Z w^{\ast}$.  Consequently, $\Lambda/\Z w^{\ast}$ is torsion free of rank $1$, and $w^{\ast}$ extends to a basis $(w^{\ast},w')$ of $\Lambda$.  See \cite{Cassels} for this and other standard facts about lattices.  Now, for all $z\in\Lambda$, there exist unique $a,b\in\Z$ such that $z=aw^{\ast}+bw'$.

If $z,z'\in\Lambda\cap R$ have the same $w'$-coordinate, then $z-z'=(a-a')w^{\ast}$ and $\|z-z'\|\le2$.  Therefore, $|a-a'|\le2/\lambda$, and for each $b\in\Z$, there exist at most $1+2/\lambda$ points of $\Lambda\cap R$ with $w'$-coordinate equal to $b$.  Next, we bound the number of values that the $w'$-coordinate can assume on $\Lambda\cap R$.  Let $\pi\colon\R^2\to\R$ be the linear form with $\pi(w^{\ast})=0$ and $\pi(w')=1$, so that $\pi(z)=b$ whenever $z=aw^{\ast}+bw'$.  Explicitly, $\pi(x,y)=(sy-tx)/\Delta$, where $\Delta\coloneq\det(w^{\ast},w')$ satisfies $|\Delta|=D$.  Since $|s|\le\lambda A$ and $|t|\le\lambda B$, we obtain
\[
\sup_{(x,y)\in R}|\pi(x,y)|\le\frac{|s|B+|t|A}{D}\le\frac{2\lambda AB}{D}.
\]
Thus, on $\Lambda\cap R$, the $w'$-coordinate is an integer in an interval of length $4\lambda AB/D$, and therefore takes at most $1+4\lambda AB/D$ values.  Multiplying the two counts and using $\lambda\le1$, we find that
\[
J+1\le\Big(1+\frac{2}{\lambda}\Big)\Big(1+\frac{4\lambda AB}{D}\Big)
\le 1+\frac{12AB}{D}+\frac{2}{\lambda}.
\]

Suppose finally that $J\ge 24AB/D$, so that $12AB/D\le J/2$.  The preceding display then yields $J\le J/2+2/\lambda$, whence $\lambda\le 4/J$.  Since $w^{\ast}=(s,t)$ with $|s|\le\lambda A$ and $|t|\le\lambda B$, it follows that
\[
|s|\le \frac{4A}{J} \qquad\text{and}\qquad |t|\le \frac{4B}{J},
\]
as desired.
\end{proof}

Call $(p,q)\in\mathcal{H}_n$ \emph{resonant} if there exist $d\in\mathcal{D}_n$ and $(k,\ell)\in\Lambda_d\cap(\mathcal{F}_d\times \mathcal{F}_d)$ such that
\begin{equation}
\label{eqn:resonance}
k\ell\ne0\qquad\text{and}\qquad |k\ell|\le\frac{1000 d \mathcal{Q}}{M(p,q)}.
\end{equation}
The pair $(p,q)$ is \emph{degenerate} if $\max\big(\gcd(p,n),\gcd(q,n)\big)\ge M(p,q)/\big(1000 \mathcal{Q}\big)$.

\begin{proposition}
\label{prop:reduction}
If $(p,q)\in\mathcal{H}_n$, $M(p,q)\ge1000 \mathcal{Q}$, and $|E(p,q)|\ge M(p,q)/2$, then $(p,q)$ is resonant or degenerate.
\end{proposition}

\begin{proof}
Write $M\coloneq M(p,q)$, and suppose that $(p,q)$ is neither resonant nor degenerate.  Since $M\ge1000\cdot2^{\omega(n)}$, we have $2^{\omega(n)+1}\le M/4$.  Lemma~\ref{lem:majd} therefore gives
\begin{equation}
\label{eqn:needbig}
\sum_{d\in\mathcal{D}_n}\mathcal{W}_d \ge \frac M2-2^{\omega(n)+1} \ge \frac M4 .
\end{equation}

For each $d\in\mathcal{D}_n$, we partition the index set of $\mathcal{W}_d$ into the families
\[
\mathcal{A}_d\coloneq \{(0,\ell)\in\Lambda_d\cap(\mathcal{F}_d\times\mathcal{F}_d)\colon \ell\ne0\},\qquad \mathcal{B}_d\coloneq \{(k,0)\in\Lambda_d\cap(\mathcal{F}_d\times\mathcal{F}_d)\colon k\ne0\},
\]
and
\[
\Lambda_d(K,L)\coloneq \{(k,\ell)\in\Lambda_d\cap(\mathcal{F}_d\times \mathcal{F}_d)\colon K\le|k|<2K,\ L\le|\ell|<2L\},
\]
where $K$ and $L$ run over the powers of $2$ that are at most $d/2$.  This is indeed a partition: every $(k,\ell)$ in the index set with $k\ell\ne0$ satisfies $1\le|k|,|\ell|\le d/2$, and hence lies in exactly one of the sets $\Lambda_d(K,L)$.  The \emph{contribution} of a family $\mathcal{G}$ is $\frac1d\sum_{(k,\ell)\in\mathcal{G}}u_d(k)v_d(\ell)$, so that $\mathcal{W}_d$ equals the sum of the contributions of its families.  The total number of families is at most $2^{\omega(n)}\big(2+(\log_2n)^2\big)$, which is at most $5 \mathcal{Q}/(1+\log n)^2$.  By \eqref{eqn:needbig} and the pigeonhole principle, there exists a family with contribution at least $(1+\log n)^2M/(20 \mathcal{Q})\geq M/(20 \mathcal{Q})$.

Suppose first that this family is $\mathcal{A}_d$ for some $d$, and let $g_q\coloneq \gcd(q,d)$.  If $(0,\ell)\in\mathcal{A}_d$, then $\ell q\equiv0\pmod d$, hence $(d/g_q)\mid\ell$.  Writing $\ell=jd/g_q$ with $j\in\mathcal{F}_{g_q}\smallsetminus\{0\}$, we have $v_d(\ell)=g_q/(2|j|)$.  In particular, $\mathcal{A}_d$ is nonempty, so $g_q\ge2$.  Since $u_d(0)=H^{(d)}_p\le h_pd/n$, Lemma~\ref{lem:harmonic} gives
\[
\frac{1}{d}\sum_{(0,\ell)\in\mathcal{A}_d}u_d(0)v_d(\ell) \le \frac{h_p}{n}\sum_{j\in\mathcal{F}_{g_q}\smallsetminus\{0\}}\frac{g_q}{2|j|} \le \frac{h_p}{n} g_q\log g_q .
\]
The left-hand side is at least $(1+\log n)^2M/(20 \mathcal{Q})$.  Since $h_p<n$ by \eqref{eqn:hsum} and $\log g_q\le1+\log n$, it follows that $g_q\ge(1+\log n)M/(20 \mathcal{Q})\ge M/(1000 \mathcal{Q})$.  As $g_q$ divides $\gcd(q,n)$, the pair $(p,q)$ is degenerate, which is a contradiction.  The family $\mathcal{B}_d$ is treated symmetrically, with $\gcd(p,d)$ in place of $g_q$.

Suppose instead that this family is $\Lambda_d(K,L)$ for some $d$, $K$, and $L$.  Let $J\coloneq \#\Lambda_d(K,L)$, so that $J\ge1$.  On this family, $u_d(k)\le d/(2K)$ and $v_d(\ell)\le d/(2L)$, so the contribution is at most $dJ/(4KL)$.  It follows that $dJ/(4KL)\ge M/(20 \mathcal{Q})$; that is,
\begin{equation}
\label{eqn:KLbound}
KL\ \le\ \frac{5 d J \mathcal{Q}}{M}.
\end{equation}
We now apply Lemma~\ref{lem:lattice} with $\Lambda=\Lambda_d$, $D=d$ (see \eqref{eqn:detLambda}), $A=2K$, and $B=2L$.  Let $J'$ denote the quantity denoted $J$ in that lemma.  The box $R$ contains $\Lambda_d(K,L)$, so $J'\ge J\ge1$.

If the first alternative of Lemma~\ref{lem:lattice} holds, then $J\le J'<24AB/D=96KL/d$.  Combined with \eqref{eqn:KLbound}, this yields $M<480 \mathcal{Q}$, which contradicts $M\ge1000 \mathcal{Q}$.  Hence the second alternative holds: since $4A/J'\le8K/J$ and $4B/J'\le8L/J$, there is $(k_0,\ell_0)\in\Lambda_d\smallsetminus\{0\}$ with $|k_0|\le8K/J$ and $|\ell_0|\le8L/J$.  Let $(k_0',\ell_0')\in\mathcal{F}_d\times\mathcal{F}_d$ be the representative of $(k_0,\ell_0)$ modulo $d$, so that $(k_0',\ell_0')\in\Lambda_d$, $|k_0'|\le|k_0|$, and $|\ell_0'|\le|\ell_0|$.  Three cases arise.

First, suppose that $(k_0',\ell_0')=(0,0)$.  Then $d$ divides both $k_0$ and $\ell_0$, so at least one of $k_0$ and $\ell_0$ is a nonzero multiple of $d$; say $|\ell_0|\ge d$, the case $|k_0|\ge d$ being symmetric.  Then $8L/J\ge|\ell_0|\ge d$, so $KL\ge L\ge Jd/8$.  Combined with \eqref{eqn:KLbound}, this yields $M\le40 \mathcal{Q}$, which is a contradiction.

Second, suppose that $k_0'\ell_0'\ne0$.  By \eqref{eqn:KLbound},
\[
|k_0'\ell_0'|\le\frac{64 KL}{J^2}\le\frac{320 d \mathcal{Q}}{J M}\le\frac{1000 d \mathcal{Q}}{M},
\]
so $(p,q)$ is resonant, which is a contradiction.

Third, suppose that exactly one of $k_0'$ and $\ell_0'$ vanishes; say $\ell_0'=0$, the case $k_0'=0$ being symmetric.  Let $g_p\coloneq \gcd(p,d)$.  Then $k_0'p\equiv0\pmod d$, so $(d/g_p)\mid k_0'$, and therefore $d/g_p\le|k_0'|\le8K$.  Now fix $\ell$ with $L\le|\ell|<2L$; there are at most $2L$ such $\ell$.  The integers $k$ with $(k,\ell)\in\Lambda_d$, if any, lie in a single residue class modulo $d/g_p$, so at most $4Kg_p/d+1$ of them satisfy $|k|<2K$; since $d/g_p\le8K$, this count is at most $12Kg_p/d$.  Hence $J\le24KLg_p/d$.  The contribution of $\Lambda_d(K,L)$ is therefore at most $dJ/(4KL)\le6g_p$, while it is at least $M/(20 \mathcal{Q})$, so $g_p\ge M/(120 \mathcal{Q})\ge M/(1000 \mathcal{Q})$.  As $g_p$ divides $\gcd(p,n)$, the pair $(p,q)$ is degenerate, which is a contradiction.

These cases are exhaustive, so the proof is complete.
\end{proof}

\subsection{Proof of Theorem~\ref{thm:Sharp}}

\begin{lemma}
\label{lem:quadcount}
Let $\varepsilon>0$, let $d\mid n$, and let $K,L,P,Q\ge1$ satisfy $KP\le n^2$ and $LQ\le n^2$.  The number of quadruples $(k,\ell,p,q)\in\Z^4$ with $K\le|k|<2K$, $L\le|\ell|<2L$, $P\le p<2P$, $Q\le q<2Q$, and $kp+\ell q\equiv0\pmod d$ is
\[
\ll_{\varepsilon}n^{\varepsilon}\Big(\frac{KLPQ}{d}+\min(KP,LQ)\Big).
\]
\end{lemma}

\begin{proof}
By symmetry, we may assume that $LQ\le KP$.  There are at most $2L$ integers $\ell$ with $L\le|\ell|<2L$ and at most $Q$ integers $q$ with $Q\le q<2Q$, hence at most $2LQ$ pairs $(\ell,q)$.  Fix one such pair.  For every quadruple counted, we have $kp+\ell q=md$ for some $m\in\Z$.  Since $|kp|<4KP$ and $|\ell q|<4LQ\le4KP$, we have $|md|<8KP$, so $m$ takes at most $16KP/d+1$ values.  Fix one such $m$.  The quadruple is now determined by the pair $(k,p)$, which satisfies $kp=md-\ell q$.  This product is nonzero, because $|k|\ge1$ and $p\ge1$, so each such pair is determined by the divisor $k$ of $md-\ell q$ together with its sign.  Therefore, there are at most $2\tau(|md-\ell q|)$ pairs $(k,p)$, and since $|md-\ell q|=|kp|<4KP\le4n^2$, the divisor bound gives $\tau(|md-\ell q|)\ll_{\varepsilon}n^{\varepsilon}$.  To complete the proof, we multiply the three counts and recall that $LQ=\min(KP,LQ)$.
\end{proof}

\begin{proof}[Proof of Theorem~\ref{thm:Sharp}]
Let $\varepsilon\in(0,1)$ and $\delta\coloneq \varepsilon/7$.  We use the bound $\mathcal{Q}\ll_{\delta}n^{2\delta}$ from \eqref{eqn:Q_small}, the consequence $n^2/\varphi(n)\ll_{\delta}n^{1+\delta}$ of \eqref{eqn:phi_lower}, and the hyperbola method in the form
\[
\#\{(x,y)\in\N^2\colon xy\le X\}\ll_{\delta}Xn^{\delta}\qquad(1\le X\le n^{10}).
\]

By Lemma~\ref{lem:usable}, it suffices to count the pairs $(p,q)\in\mathcal{H}_n$ with $S(p,q)\le2$.  For such a pair, $|E(p,q)|\ge M(p,q)-2$, because $E(p,q)=S(p,q)-M(p,q)$.  If $M(p,q)<1000 \mathcal{Q}$, then $pq\le h_ph_q=M(p,q)n^2/\varphi(n)\ll_{\delta}n^{1+3\delta}$, so the hyperbola bound shows that there are $\ll_{\delta}n^{1+4\delta}$ such pairs.  We may therefore assume that $M\coloneq M(p,q)\ge1000 \mathcal{Q}$.  Since $\mathcal{Q}\ge2$ by \eqref{eqn:Q_small}, we have $M\ge4$ and $|E(p,q)|\ge M-2\ge M/2$. Proposition~\ref{prop:reduction} now shows that $(p,q)$ is resonant or degenerate.

Suppose that $(p,q)$ is degenerate; say $g\coloneq \gcd(q,n)\ge M/(1000 \mathcal{Q})$, the case $\gcd(p,n)\ge M/(1000 \mathcal{Q})$ being symmetric.  Write $q=gq'$ with $q'\in\N$.  Since $h_q\ge q=gq'$, the bound $g\ge h_ph_q\varphi(n)/(1000 \mathcal{Q}n^2)$ implies that
\[
g\ \ge\ \frac{h_p g q' \varphi(n)}{1000 \mathcal{Q} n^2}, \qquad\text{and hence}\qquad pq'\ \le\ h_pq'\ \le\ \frac{1000 \mathcal{Q} n^2}{\varphi(n)}\ \ll_{\delta}\ n^{1+3\delta}.
\]
By the hyperbola bound, the pair $(p,q')$ takes $\ll_{\delta}n^{1+4\delta}$ values.  The integer $g$ is one of the $\tau(n)\ll_{\delta}n^{\delta}$ divisors of $n$, and the pair $(g,q')$ determines $q$.  Hence, there are $\ll_{\delta}n^{1+5\delta}$ degenerate pairs.

Finally, suppose that $(p,q)$ is resonant.  Cover $\mathcal{H}_n$ by the boxes $P\le p<2P$, $Q\le q<2Q$, where $P$ and $Q$ run over the powers of $2$ that are less than $n$; there are at most $(\log_2n)^2$ boxes.  Fix a box, and consider the resonant pairs that it contains.  For such a pair, $h_p\ge P$ and $h_q\ge Q$, so $M\ge PQ\varphi(n)/n^2$.  Define $T_d\coloneq 1000 d \mathcal{Q}n^2/(PQ\varphi(n))$.  By \eqref{eqn:resonance}, the pair admits $d\in\mathcal{D}_n$ and $(k,\ell)\in\Lambda_d\cap(\mathcal{F}_d\times\mathcal{F}_d)$ with $k\ell\ne0$ and $|k\ell|\le1000 d \mathcal{Q}/M\le T_d$.  Choose the powers of $2$ with $K\le|k|<2K$ and $L\le|\ell|<2L$.  Then $KL\le|k\ell|\le T_d$, and $K,L\le d/2\le n$ because $|k|,|\ell|\le d/2$.  In particular, $KP\le n^2$ and $LQ\le n^2$, so Lemma~\ref{lem:quadcount} applies.  Summing over the possible parameters, the number of resonant pairs in the box is at most
\[
\ll_{\delta}n^{\delta}\sum_{d\in\mathcal{D}_n}\ \sum_{\substack{K,L\ \textup{powers of }2\\ K,L\le n,\ KL\le T_d}} \Big(\frac{KLPQ}{d}+\min(KP,LQ)\Big).
\]
The first term satisfies $KLPQ/d\le T_dPQ/d=1000 \mathcal{Q}n^2/\varphi(n)\ll_{\delta}n^{1+3\delta}$, independently of $d$, $P$, and $Q$.  The second term satisfies
\[
\min(KP,LQ)^2\le KLPQ\le T_dPQ \le 1000 \mathcal{Q}n \cdot n^2/\varphi(n)\ll_{\delta}n^{2+3\delta},
\]
so that $\min(KP,LQ)\ll_{\delta}n^{1+2\delta}$.  The number of triples $(d,K,L)$ appearing in the sum is at most $2^{\omega(n)}(1+\log_2n)^2\ll_{\delta}n^{2\delta}$.  Hence the box contributes $\ll_{\delta}n^{\delta}\cdot n^{2\delta}\cdot n^{1+3\delta}=n^{1+6\delta}$, and summing over the $(\log_2n)^2$ boxes gives $\ll_{\delta}n^{1+7\delta}$ resonant pairs.

Adding the estimates from the three cases and recalling that $7\delta=\varepsilon$, we obtain $\#\mathcal{L}_n\ll_{\varepsilon}n^{1+\varepsilon}$.  Since $\#\mathcal{H}_n\gg n^2$ by Lemma~\ref{lem:sizeHn}, the theorem follows.
\end{proof}

\begin{remark}
\label{rem:optimal}
The exponent $-1$ is optimal for arguments resting on Proposition~\ref{prop:LNZ}.  From the point of view of discrepancy, the reason is that no equidistribution statement can detect a box that is expected to contain one point.  Indeed, if $h_ph_q\le n^2/\varphi(n)$, then $M(p,q)\le1$.  A point set may miss such a box altogether, in which case $S(p,q)\le2$ and Proposition~\ref{prop:LNZ} yields no information.  Since $n^2/\varphi(n)\ge n$, there are $\gg n$ such pairs in $\mathcal{H}_n$.
\end{remark}

\section*{Declaration of generative AI and AI-assisted technologies}\label{sec:AxiomProver}

Our initial journal submission proved a theorem weaker than Theorem~\ref{thm:Sharp}.  The final theorem statement and proof arose from human-LLM interaction while investigating some suggestions in our referee report. 

To make the scope precise, we emphasize that the AI system was not asked to reprove deep external
theorems in the theory of translation surfaces, Veech theory,
and the Mirzakhani-Wright rank obstruction theorem and its number theoretic
formulation Proposition~\ref{prop:LNZ}. 
The proof of Theorem~\ref{thm:Sharp} was autonomously formalized assuming these facts in Lean \textbf{4.26.0} (see \cite{Mathlib2020, Lean}) using AxiomProver.  All relevant files are posted in the repository  
\begin{center}
\url{https://github.com/AxiomMath/paucity}.
\end{center}

\end{document}